\documentclass[reqno]{amsart}
\usepackage{amsmath,amssymb,amsthm}
\usepackage{hyperref}
\usepackage{geometry}
\usepackage{amsmath,amsfonts,amssymb,amsthm}
\usepackage{mathtools}
\usepackage{verbatim}
\usepackage{xcolor}
\usepackage{xfrac}
\usepackage{tikz-cd}

\newtheorem{Conjecture}{Conjecture}
\newtheorem{Question}[Conjecture]{Question}
\newtheorem{lem}{Lemma}[section]
\newtheorem{teo}[lem]{Theorem}
\newtheorem{pro}[lem]{Proposition}
\newtheorem{cor}[lem]{Corollary}

\newtheorem{rem}[lem]{Remark}

 \newtheorem{exam}[lem]{Example}

\newtheorem*{rem*}{Remark}
\newtheorem*{teo*}{Theorem}

\newcounter{claimcounter}
\numberwithin{claimcounter}{lem}

\usepackage{scalerel,stackengine}
\stackMath
\newcommand\reallywidehat[1]{%
\savestack{\tmpbox}{\stretchto{%
\scaleto{%
	\scalerel*[\widthof{\ensuremath{#1}}]{\kern-.6pt\bigwedge\kern-.6pt}%
	{\rule[-\textheight/2]{1ex}{\textheight}}
	}{\textheight}%
	}{0.5ex}}%
	\stackon[1pt]{#1}{\tmpbox}%
}

\newcommand{\myeq}[1]{\ensuremath{\stackrel{\text{#1}}{=}}}

\DeclareMathOperator{\GL}{GL}

\DeclareMathOperator{\im}{Im}

\newcommand{\F}{\mathbb{F}}
\newcommand{\M}{\mathcal{M}}

\newcommand{\Q}{\mathcal{Q}}

\renewcommand{\L}[0]{\mathcal{L}}

\DeclareMathOperator{\CR}{\rm{Crit}}

\renewcommand{\emph}[1]{{\bf #1}}
\date{July 2026}

  \title[Homomorphisms to the random module of finite dimension]{Average Numbers of Homomorphisms to Random Modules over Free Group Algebras}

\author{J. de la Nuez G\'onzalez}

\address{Korea Institute For Advanced Study}
 
\email{jnuezgonzalez@gmail.com}

\author{A. Jaikin-Zapirain}

\address{Instituto de Ciencias Matem\'aticas, CSIC-UAM-UC3M-UCM}
 
\email{andrei.jaikin@icmat.es}

\begin{document}

\begin{abstract}
Let $F$ be a finitely generated free group, and let $L$ be a finitely presented
$\mathbb{F}_q[F]$-module. We study the average number $\Lambda_L(n)$ of
$\mathbb{F}_q[F]$-module homomorphisms from $L$ to an
$\mathbb{F}_q[F]$-module of dimension $n$ over $\mathbb{F}_q$.

We show that, for all sufficiently large $n$, the quantity $\Lambda_L(n)$ is
given by a rational function of $q^n$ and satisfies
\[
\Lambda_L(n)
=
q^{\chi(L)n}
+
\sum_{N\in A(L)}
q^{\chi(L/N)n}\bigl(1+O(q^{-n})\bigr),
\]
where $\chi(L)$ denotes the Euler characteristic of $L$, and $A(L)$ is the
set of nonzero $\mathbb{F}_q[F]$-submodules $N$ of $L$ that have no nonzero
free quotients.

Our proof is based on a theory of partial modules that may be of independent interest and have further applications.
\end{abstract}

\maketitle

\section{Introduction}

Let $F$ be the free group on the free generating set
$\{x_1,\ldots,x_d\}$, and let
$h=(h_1,\ldots,h_t)$ be a tuple of elements of $F$.
We say that a tuple of elements of $F$ is \textbf{primitive} if it is
contained in a free generating set of $F$.

Given a finite group $G$, we define a map
\[
h_G\colon G^d\longrightarrow G^t
\]
by
\[
h_G(g_1,\ldots,g_d)
=
\bigl(h_1(g_1,\ldots,g_d),\ldots,
h_t(g_1,\ldots,g_d)\bigr).
\]
In other words, $h_G(g_1,\ldots,g_d)$ is the image of the tuple $h$
under the homomorphism $F\to G$ sending $x_i$ to $g_i$.
If $h$ is primitive, then $h_G$ is \textbf{measure-preserving}, meaning
that all its fibers have the same cardinality, for every finite group
$G$. We are interested in the converse implication.

\begin{Question}\label{separablefamily}
Let $F$ be a free group, and let $h$ be a tuple of elements of $F$.
For which families $\mathcal F$ of finite groups does the following
implication hold?
\[
h_G \text{ is measure-preserving for every }G\in\mathcal F
\quad\Longrightarrow\quad
h\text{ is primitive}.
\]
\end{Question}

A positive answer for the family of symmetric groups was given by
Puder and Parzanchevski in \cite{PP15}. They considered the average
number of common fixed points of $\{h_1,\ldots,h_t\}$ in a permutation
representation of $F$ of degree $n$. This number depends only on the
subgroup
\[
H=\langle h_1,\ldots,h_t\rangle
\]
and is denoted by $\Phi_{F,H}(n)$. They proved that, for all sufficiently
large $n$, the function $\Phi_{F,H}(n)$ agrees with a rational function
in $n$ and that
\begin{equation}\label{permutation}
\Phi_{F,H}(n)
=
n^{1-\operatorname{rk}(H)}
+
\sum_{J\in A(H)}
n^{1-\operatorname{rk}(J)}
\bigl(1+O(n^{-1})\bigr),
\end{equation}
where $A(H)$ is the set of nontrivial algebraic extensions $J$ of $H$;
that is, $H\lneq J$ and $H$ is not contained in a proper free factor
of $J$.

The case of Question~\ref{separablefamily} in which $\mathcal F$ is
the family of all finite soluble groups was studied in \cite{Ja22}.
In particular, it was shown that the answer is positive when the size
of the tuple $h$ is equal to the rank of $F$.

In \cite{EPS21}, Ernst-West, Puder, and Seidel considered the case in
which $h$ consists of a single nontrivial word $w$ and
\[
\mathcal F=\mathrm{GL}(q)
:=
\{\mathrm{GL}_n(q)\mid n\geq 1\},
\]
where $q$ is fixed. They used an approach similar to that of
\cite{PP15} and studied the average number of fixed vectors of $w$ in
a random representation of $F$ in $\mathrm{GL}_n(q)$. They denoted
this quantity by $\mathbb E_w[\mathrm{fix}]$.

They proved that, for all sufficiently large $n$,
$\mathbb E_w[\mathrm{fix}]$ is a rational function of $q^n$ and that
\begin{equation}\label{average}
\mathbb E_w[\mathrm{fix}]
=
c_{l(w),q}+O(q^{-n}),
\end{equation}
where
\[
l(w)
=
\bigl|C_F(w)/\langle w\rangle\bigr|
\]
and $c_{l,q}$ is the number of ideals of
\[
\mathbb F_q[x]/(x^l-1).
\]
Moreover, if $w$ is not a proper power, equivalently,
$C_F(w)=\langle w\rangle$, they also proved that
\begin{equation}\label{crit2}
\mathbb E_w[\mathrm{fix}]
=
2+|\CR_q^2(w)|q^{-n}+O(q^{-2n}),
\end{equation}
where $\CR_q^2(w)$ is the set of rank-two right submodules
$I\leq \mathbb F_q[F]$ containing $w-1$ as an imprimitive element;
that is,
\[
(w-1)\mathbb F_q[F]
\]
is not a direct summand of $I$.

In particular, this yields a positive answer to
Question~\ref{separablefamily} when $F$ has rank two and
$\mathcal F=\mathrm{GL}(q)$.

A homomorphism
\[
\phi\colon F\longrightarrow \mathrm{GL}_n(q)
\]
induces the structure of a right $\mathbb F_q[F]$-module on
$\mathbb F_q^n$, which we denote by $(\mathbb F_q^n)_\phi$.
Given a finitely generated right $\mathbb F_q[F]$-module $L$, define
\[
\Lambda_L(n)
=
\frac{1}{|\mathrm{GL}_n(q)|^d}
\sum_{\phi\colon F\to\mathrm{GL}_n(q)}
\left|
\operatorname{Hom}_{\mathbb F_q[F]}
\bigl(L,(\mathbb F_q^n)_\phi\bigr)
\right|.
\]
Thus, $\Lambda_L(n)$ is the average number of
$\mathbb F_q[F]$-module homomorphisms from $L$ to a random
$\mathbb F_q[F]$-module whose underlying $\mathbb F_q$-vector space
has dimension $n$.

Observe that, if $w\in F$, then
\begin{equation}\label{interpretation}
\mathbb E_w[\mathrm{fix}]
=
\Lambda_{\mathbb F_q[F]/(w-1)\mathbb F_q[F]}(n).
\end{equation}

In this paper, we study $\Lambda_L(n)$ for an arbitrary finitely
presented right $\mathbb F_q[F]$-module $L$. Our approach is strongly
influenced by \cite{PP15} and \cite{EPS21}.

A right $\mathbb F_q[F]$-module is called \textbf{algebraic} if it has
no nontrivial free direct summand. Equivalently, it does not admit a
surjection onto a nontrivial free module. For a right
$\mathbb F_q[F]$-module $L$, we denote by $A(L)$ the set of its
nontrivial algebraic submodules.

Since every submodule of a free $\mathbb F_q[F]$-module is free, the
module $L$ has a maximal free quotient. Moreover, the smallest
submodule $N\leq L$ such that $L/N$ is free is algebraic. Consequently,
$L$ is free if and only if $A(L)$ is empty.

If $L$ is finitely presented, then
\[
L\cong \mathbb F_q[F]^r/U,
\]
where
\[
U\cong \mathbb F_q[F]^s
\]
is free. The Euler characteristic of $L$ is defined by
\[
\chi(L)=r-s.
\]
This definition is independent of the chosen presentation of $L$;
see \cite{Co64}.

\begin{teo}\label{teo:average}
Let $L$ be a finitely presented right $\mathbb F_q[F]$-module. Then,
for all sufficiently large $n$, the function $\Lambda_L(n)$ is a
rational function of $q^n$.

Moreover, the set $A(L)$ is finite and
\begin{equation}\label{representation}
\Lambda_L(n)
=
q^{\chi(L)n}
+
\sum_{N\in A(L)}
q^{\chi(L/N)n}
\bigl(1+O(q^{-n})\bigr).
\end{equation}
\end{teo}

Let us relate this result to the asymptotic formula
\eqref{average}. Define
\[
f(L)
=
\begin{cases}
\displaystyle
\min\{\chi(N)\mid N\in A(L)\},
& \text{if }A(L)\neq\varnothing,\\[6pt]
+\infty,
& \text{if }L\text{ is free},
\end{cases}
\]
and set
\[
\CR(L)
=
\{N\in A(L)\mid \chi(N)=f(L)\}.
\]
Since
\[
\chi(L/N)=\chi(L)-\chi(N),
\]
the conclusion of Theorem~\ref{teo:average} can be reformulated as
\begin{equation}\label{averagecrit}
\Lambda_L(n)
=
q^{\chi(L)n}
+
|\CR(L)|q^{(\chi(L)-f(L))n}
\bigl(1+O(q^{-n})\bigr).
\end{equation}

If $1\neq w\in F$ and
\[
L=\mathbb F_q[F]/(w-1)\mathbb F_q[F],
\]
then
\[
\chi(L)=f(L)=0
\qquad\text{and}\qquad
|\CR(L)|=c_{l(w),q}-1;
\]
see \cite[Lemma~4.2]{EPS21}. Therefore, by
\eqref{interpretation}, formula \eqref{averagecrit} implies
\eqref{average}.

\section*{Acknowledgments}

The work is partially supported by the grants  {\it PID2020-114032GB-I00}, {\it PID2024-155800NB-C33} and {\it EUR2025-164928} of the Ministry of Science, Innovation and Universities of Spain. The first author has been supported by the Mid-Career Researcher Program (RS-2023-00278510) through the National Research Foundation funded by the government of Korea and by the KIAS individual grant SP084001.

\section{Partial modules}
 
Let $K$ be a field, let $F$ be the free group freely generated by a set
$X$, and let $R=K[F]$ be the group algebra of $F$ over $K$. In this
section, we introduce the notion of a partial $R$-module and describe
its main properties.

For simplicity of exposition, we assume that $X$ is finite, although
most of the results can be easily adapted to the case in which $X$ is
infinite.

\subsection{Schreier transversals}

In this subsection, we recall the notions of Schreier transversals and
Schreier generators introduced by Lewin in \cite{Lew69}.

Let $E$ be a free right $R$-module with free basis
$\{e_\lambda\}_{\lambda\in\Lambda}$. The elements of the set
\[
\{e_\lambda\cdot f\mid f\in F,\ \lambda\in\Lambda\}
\]
are called \textbf{monomials} of $E$.

Suppose that $f\in F$ is written in reduced form as
\[
f=x_1^{\varepsilon_1}\cdots x_\ell^{\varepsilon_\ell},
\qquad
x_i\in X,\quad \varepsilon_i\in\{\pm1\}.
\]
For $0\leq i\leq \ell$, we set
\[
{}^{(i)}\!f
=
x_1^{\varepsilon_1}\cdots x_i^{\varepsilon_i}
\qquad\text{and}\qquad
f^{(i)}
=
x_{i+1}^{\varepsilon_{i+1}}\cdots
x_\ell^{\varepsilon_\ell}.
\]
In particular,
\[
{}^{(0)}\!f=f^{(\ell)}=1
\qquad\text{and}\qquad
{}^{(\ell)}\!f=f^{(0)}=f.
\]
As usual, $\ell$ is called the \textbf{length} of $f$. The length of
the monomial $e_\lambda\cdot f$ is defined to be the length of $f$.

The monomials form a $K$-basis of $E$. If
\[
u=\sum_{i=1}^r k_i m_i,
\]
where $0\neq k_i\in K$ and the $m_i$ are distinct monomials, then the
\textbf{length} of $u$ is the maximum of the lengths of the monomials
$m_i$.

Let $D$ be an $R$-submodule of $E$. A
\textbf{partial Schreier basis for $E$ modulo $D$ with respect to
$\{e_\lambda\}_{\lambda\in\Lambda}$} is a set $B$ of monomials of $E$
such that:
\begin{enumerate}
    \item $B$ is $K$-linearly independent modulo $D$;
    \item if $b=e_\lambda\cdot f$ belongs to $B$, then
    \[
    e_\lambda\cdot{}^{(i)}\!f\in B
    \]
    for every $0\leq i\leq \ell$.
\end{enumerate}

A partial Schreier basis $B$ is called a \textbf{Schreier basis} if,
in addition, its image spans $E/D$ as a $K$-vector space. In this
case, the $K$-linear span
\[
T=\langle B\rangle_K
\]
is a transversal for $D$ in $E$; that is,
\[
E=T\oplus D
\]
as $K$-vector spaces. We denote by
\[
\pi_B\colon E\longrightarrow T
\]
the $K$-linear projection determined by the condition
\[
u+D=\pi_B(u)+D
\]
for every $u\in E$.

Every partial Schreier basis can be extended to a Schreier basis; see
\cite[Lemma~1]{Lew69}.

\begin{teo}[{\cite[Theorem~1]{Lew69}}]\label{schreier}
Let $E$ be a free right $R$-module with basis
$\{e_\lambda\}_{\lambda\in\Lambda}$, let $D$ be an $R$-submodule of
$E$, and let $B$ be a Schreier basis for $E$ modulo $D$. Consider the
set
\[
\begin{split}
S={}&
\bigl\{
e_\lambda-\pi_B(e_\lambda)
\mathrel{}\big|\mathrel{}
\lambda\in\Lambda,\ 
e_\lambda-\pi_B(e_\lambda)\neq 0
\bigr\}
\\
&\cup
\bigl\{
b\cdot x-\pi_B(b\cdot x)
\mathrel{}\big|\mathrel{}
b\in B,\ x\in X,\ 
b\cdot x-\pi_B(b\cdot x)\neq 0
\bigr\}.
\end{split}
\]
Then $D$ is freely generated as a right $R$-module by the set $S$.
\end{teo}

    \subsection{The definition of partial $R$-modules}

A \textbf{(right) partial $R$-module} is a tuple
\[
\mathcal M
=
\bigl(M,(\phi_x\colon U_x\to V_x)_{x\in X}\bigr),
\]
where $M$ is a $K$-vector space, $U_x$ and $V_x$ are $K$-subspaces of
$M$, and $\phi_x\colon U_x\to V_x$ is a $K$-linear isomorphism for
every $x\in X$. The maps $\phi_x$ represent partial multiplication by
the elements of $X$:
\[
u\cdot x=\phi_x(u)
\qquad\text{for }u\in U_x.
\]
We set
\[
m(\mathcal M)=\dim_K M
\qquad\text{and}\qquad
m_x(\mathcal M)=\dim_K U_x
\quad (x\in X).
\]
We call $M$ the \textbf{base space} of $\mathcal M$.

Let
\[
\mathcal N
=
\bigl(N,(\psi_x\colon W_x\to Z_x)_{x\in X}\bigr)
\]
be another partial $R$-module. We say that $\mathcal N$ is a
\textbf{partial submodule} of $\mathcal M$ if $N$ is a $K$-subspace of
$M$ and, for every $x\in X$,
\[
W_x=N\cap U_x,\qquad Z_x=N\cap V_x,
\]
and $\psi_x$ is the restriction of $\phi_x$ to $W_x$. In particular,
the restriction of $\phi_x$ maps $N\cap U_x$ isomorphically onto
$N\cap V_x$.

In this situation, there is an associated \textbf{partial quotient
module}
\[
\mathcal M/\mathcal N
=
\bigl(M/N,
(\overline{\phi}_x\colon\overline U_x\to\overline V_x)_{x\in X}
\bigr),
\]
where
\[
\overline U_x=(U_x+N)/N,\qquad
\overline V_x=(V_x+N)/N,
\]
and
\[
\overline{\phi}_x(u+N)=\phi_x(u)+N
\qquad (u\in U_x).
\]

A subset $S\subseteq M$ \textbf{generates} $\mathcal M$ if $\mathcal
M$ is the smallest partial submodule of itself containing $S$. We
denote this by
\[
\mathcal M=\langle S\rangle.
\]

If
\[
\mathcal L
=
\bigl(L,(\phi_x\colon L\cap U_x\to L\cap V_x)_{x\in X}\bigr)
\]
and
\[
\mathcal N
=
\bigl(N,(\phi_x\colon N\cap U_x\to N\cap V_x)_{x\in X}\bigr)
\]
are partial $R$-submodules of $\mathcal M$, then
\[
\mathcal L\cap\mathcal N
:=
\bigl(
L\cap N,
(\phi_x\colon L\cap N\cap U_x\to L\cap N\cap V_x)_{x\in X}
\bigr)
\]
is also a partial $R$-submodule. However, the base space of
$\langle L,N\rangle$ may be strictly larger than $L+N$.

A \textbf{homomorphism} between two partial $R$-modules
\[
\mathcal M
=
\bigl(M,(\phi_x\colon U_x\to V_x)_{x\in X}\bigr)
\quad\text{and}\quad
\mathcal N
=
\bigl(N,(\psi_x\colon W_x\to Z_x)_{x\in X}\bigr)
\]
is a $K$-linear map $\alpha\colon M\to N$ such that, for every
$x\in X$,
\[
\alpha(U_x)\subseteq W_x
\]
and
\[
\psi_x(\alpha(u))=\alpha(\phi_x(u))
\qquad\text{for every }u\in U_x.
\]
By an abuse of notation, we also write
\[
\alpha\colon\mathcal M\to\mathcal N.
\]
Observe that
\[
\ker\alpha
=
\bigl(
\ker\alpha,
(\phi_x\colon\ker\alpha\cap U_x
\to\ker\alpha\cap V_x)_{x\in X}
\bigr)
\]
is a partial submodule of $\mathcal M$. In general, however,
$\mathcal M/\ker\alpha$ need not be isomorphic to a partial submodule
of $\mathcal N$.

Given a partial $R$-module
\[
\mathcal M
=
\bigl(M,(\phi_x\colon U_x\to V_x)_{x\in X}\bigr),
\]
a homomorphism of partial $R$-modules
\[
\alpha_{\mathcal M}\colon\mathcal M\to U(\mathcal M)
\]
is called a \textbf{universal enveloping map} of $\mathcal M$ if
$U(\mathcal M)$ is a right $R$-module and the following universal
property holds: for every homomorphism
\[
\beta\colon\mathcal M\to L
\]
from $\mathcal M$ to a right $R$-module $L$, there exists a unique
$R$-homomorphism
\[
\gamma\colon U(\mathcal M)\to L
\]
such that
\[
\beta=\gamma\circ\alpha_{\mathcal M}.
\]
We call $U(\mathcal M)$ the \textbf{universal enveloping module} of
$\mathcal M$.

The construction of $U(\mathcal M)$ is straightforward. Let
$\{e_\lambda\mid\lambda\in\Lambda\}$ be a $K$-basis of $M$, and let
$E$ be the free right $R$-module with free basis
$\{e_\lambda\mid\lambda\in\Lambda\}$. We also denote by $M$ the
$K$-linear span of these elements in $E$. Let $D$ be the
$R$-submodule of $E$ generated by
\[
\{m\cdot x-\phi_x(m)\mid x\in X,\ m\in U_x\}.
\]
Then
\begin{equation}\label{constrenvelop}
U(\mathcal M)=E/D,
\end{equation}
and the map
\[
\alpha_{\mathcal M}\colon\mathcal M\to U(\mathcal M)
\]
induced by the inclusion $M\hookrightarrow E$ is a universal
enveloping map. This proves the existence assertion in the following
proposition.

\begin{pro}\label{universalenveloping}
Let $\mathcal M$ be a partial $R$-module. Then its universal
enveloping module $U(\mathcal M)$ exists and is unique up to a unique
$R$-isomorphism compatible with the universal enveloping maps.
Moreover, $\alpha_{\mathcal M}$ is injective on the base space $M$.
If $\mathcal M$ is finite-dimensional over $K$, then
\[
\chi\bigl(U(\mathcal M)\bigr)
=
m(\mathcal M)-\sum_{x\in X}m_x(\mathcal M).
\]
\end{pro}

\begin{proof}
Uniqueness follows immediately from the universal property.

To prove that $\alpha_{\mathcal M}$ is injective, extend each
isomorphism
\[
\phi_x\colon U_x\to V_x
\]
to a $K$-linear automorphism of $M$. Since $F$ is freely generated by
$X$, these automorphisms define a right $R$-module structure on the
$K$-vector space $M$ that extends all the partial multiplications
$\phi_x$. The identity map from $\mathcal M$ to this $R$-module
extends, by the universal property, to an $R$-homomorphism from
$U(\mathcal M)$. It follows that $\alpha_{\mathcal M}$ is injective.

Now assume that $\mathcal M$ is finite-dimensional over $K$. Recall
that $U(\mathcal M)=E/D$. By the construction of $D$ and
Theorem~\ref{schreier},
\[
\operatorname{rk}(D)
=
\sum_{x\in X}m_x(\mathcal M).
\]
Since $\operatorname{rk}(E)=m(\mathcal M)$, we obtain
\[
\chi\bigl(U(\mathcal M)\bigr)
=
m(\mathcal M)-\sum_{x\in X}m_x(\mathcal M).
\]
\end{proof}

In view of Proposition~\ref{universalenveloping}, for every
finite-dimensional partial $R$-module $\mathcal M$ we define
\[
\chi(\mathcal M)
:=
\chi\bigl(U(\mathcal M)\bigr)
=
m(\mathcal M)-\sum_{x\in X}m_x(\mathcal M).
\]
We shall not distinguish between the base space $M$ and its image
$\alpha_{\mathcal M}(M)$ in $U(\mathcal M)$.

\begin{pro}\label{presentationpartial}
For every finitely presented right $R$-module $L$, there exists a
finite-dimensional partial $R$-module $\mathcal M$ such that
\[
L\cong U(\mathcal M).
\]
Moreover,
\[
\chi(L)
=
m(\mathcal M)-\sum_{x\in X}m_x(\mathcal M).
\]
Given a finite-dimensional $K$-subspace $V$ of $L$, one may choose
$\mathcal M$, with base space $M$, together with an isomorphism
\[
\tau\colon L\to U(\mathcal M)
\]
such that
\[
\tau(V)\subseteq M.
\]
\end{pro}

\begin{proof}
Write
\[
L=E/D,
\]
where $E$ is a free right $R$-module with finite free basis
$\{e_\lambda\}_{\lambda\in\Lambda}$ and $D$ is a finitely generated
submodule of $E$. Without loss of generality, we may assume that the
elements $\{e_\lambda\}_{\lambda\in\Lambda}$ are $K$-linearly
independent modulo $D$. Let $B$ be a Schreier basis for $E$ modulo
$D$ containing $\{e_\lambda\}_{\lambda\in\Lambda}$.

By Theorem~\ref{schreier}, the set
\[
B_0
=
\left\{
b\in B\ \middle|\ 
\text{there exists }x\in X\text{ such that }
b\cdot x-\pi_B(b\cdot x)\neq 0
\right\}
\]
is finite. Hence, there exists a finite partial Schreier basis
$B_1\subseteq B$ containing
\[
\{e_\lambda\}_{\lambda\in\Lambda}\cup B_0
\]
such that, if $M$ denotes the $K$-linear span of $B_1$, then
\[
\pi_B(b\cdot x)\in M
\qquad\text{for every }b\in B_0\text{ and }x\in X.
\]
If a finite-dimensional subspace $V\leq L$ has been specified, we may
also choose $B_1$ large enough that $M$ contains representatives of a
$K$-basis of $V$.

For each $x\in X$, set
\[
U_x
=
\{m\in M\mid \pi_B(m\cdot x)\in M\}
\]
and define
\[
\phi_x\colon U_x\to M,
\qquad
\phi_x(m)=\pi_B(m\cdot x).
\]
The map $\phi_x$ is $K$-linear. If $m\in\ker\phi_x$, then
$m\cdot x\in D$, and hence $m\in D$. Since $M$ is spanned by elements
of $B$ and $B$ is linearly independent modulo $D$, it follows that
$m=0$. Thus, $\phi_x$ is injective. Put
\[
V_x=\operatorname{im}\phi_x
\]
and
\[
\mathcal M
=
\bigl(M,(\phi_x\colon U_x\to V_x)_{x\in X}\bigr).
\]

By the universal property of
\[
\alpha_{\mathcal M}\colon\mathcal M\to U(\mathcal M),
\]
the natural map $M\to L$ extends to an $R$-homomorphism
\[
\varphi\colon U(\mathcal M)\to L.
\]
By the Schreier property of $B_1$, the partial module $\mathcal M$ is
generated by $\{e_\lambda\}_{\lambda\in\Lambda}$. Consequently,
$U(\mathcal M)$ is generated by these elements as an $R$-module.

By the construction of $M$ and Theorem~\ref{schreier}, the
$R$-homomorphism
\[
E\to U(\mathcal M),
\qquad
e_\lambda\mapsto e_\lambda,
\]
vanishes on $D$. It therefore induces an $R$-homomorphism
\[
\psi\colon E/D\to U(\mathcal M).
\]
It is immediate that $\psi$ and $\varphi$ are mutually inverse.
Hence,
\[
L\cong U(\mathcal M).
\]

The formula for the Euler characteristic follows from
Proposition~\ref{universalenveloping}. The final assertion follows by
choosing $B_1$ so that the representatives of the prescribed
subspace $V$ lie in $M$.
\end{proof}

The following consequence illustrates the usefulness of
Proposition~\ref{presentationpartial} and generalizes
\cite[Theorem~4.3]{RR94}.

\begin{cor}\label{residually-finite-dimensional}
Let $L$ be a finitely presented right $R$-module, and let $V$ be a
finite-dimensional $K$-subspace of $L$. Then there exists a
finite-dimensional quotient $\overline L$ of $L$ such that the
restriction of the quotient map to $V$ is injective. In particular,
$L$ is residually finite-dimensional.
\end{cor}

\begin{proof}
Choose
\[
\mathcal M
=
\bigl(M,(\phi_x\colon U_x\to V_x)_{x\in X}\bigr)
\]
as in Proposition~\ref{presentationpartial}, together with an
isomorphism
\[
\tau\colon L\to U(\mathcal M)
\]
such that $\tau(V)\subseteq M$.

Extend each partial isomorphism $\phi_x$ to a $K$-linear automorphism
of $M$. These extensions equip the $K$-vector space $M$ with the
structure of a right $R$-module, which we denote by $\overline L$.
The identity map on $M$ induces a surjective $R$-homomorphism
\[
\gamma\colon U(\mathcal M)\to\overline L.
\]
The restriction of $\gamma\circ\tau$ to $V$ is injective.
\end{proof}

\begin{rem}
We shall prove in Proposition~\ref{freealgebraic} that the kernel of
the map
\[
\gamma\circ\tau\colon L\to\overline L
\]
constructed above is free.
\end{rem}

Let
\[
\mathcal M
=
\bigl(M,(\phi_x\colon U_x\to V_x)_{x\in X}\bigr)
\]
be a partial $R$-module, and let $N$ be a $K$-subspace of $M$. We
define
\[
\mathcal M^N
=
\bigl(N,(\phi_x\colon U_x^N\to V_x^N)_{x\in X}\bigr),
\]
where
\[
U_x^N
=
\{u\in N\cap U_x\mid \phi_x(u)\in N\}
\qquad\text{and}\qquad
V_x^N=\phi_x(U_x^N).
\]
The inclusion map
\[
\mathcal M^N\to\mathcal M
\]
is a homomorphism of partial $R$-modules. Notice that
$\mathcal M^N$ need not be a partial submodule of $\mathcal M$.

Let
\[
\gamma\colon\mathcal M\to\mathcal N
\]
be a homomorphism of partial $R$-modules. It induces a unique
$R$-homomorphism
\[
U(\gamma)\colon U(\mathcal M)\to U(\mathcal N)
\]
such that
\[
U(\gamma)\circ\alpha_{\mathcal M}
=
\alpha_{\mathcal N}\circ\gamma.
\]
We now give a criterion for $U(\gamma)$ to be an isomorphism. We state
it for finite-dimensional partial modules, although it can be
extended to the general setting.

\begin{pro}\label{isomorphism}
Let
\[
\mathcal M
=
\bigl(M,(\phi_x\colon U_x\to V_x)_{x\in X}\bigr)
\quad\text{and}\quad
\mathcal N
=
\bigl(N,(\psi_x\colon W_x\to Z_x)_{x\in X}\bigr)
\]
be finite-dimensional partial $R$-modules, and let
\[
\gamma\colon\mathcal M\to\mathcal N
\]
be a homomorphism. Then $U(\gamma)$ is an isomorphism if and only if
the following two conditions hold:
\begin{enumerate}
\item[(a)]
\[
\ker\gamma=\{0\}
\qquad\text{and}\qquad
\chi(\mathcal N)=\chi(\mathcal M).
\]

\item[(b)] There exists a sequence of $K$-subspaces
\[
\gamma(M)=K_0\subsetneq K_1\subsetneq\cdots\subsetneq K_s=N
\]
such that, for every $i=1,\ldots,s$,
\begin{itemize}
\item
\[
\dim_K(K_i/K_{i-1})=1;
\]

\item
\[
\chi(\mathcal N^{K_i})
=
\chi(\mathcal N^{K_{i-1}});
\]

\item there exists $x\in X$ for which one of the following holds:
\begin{enumerate}
\item[(i)] there exists $v\in K_{i-1}\cap W_x$ such that
\[
u=\psi_x(v)\in K_i\setminus K_{i-1};
\]

\item[(ii)] there exists $v\in K_{i-1}\cap Z_x$ such that
\[
u=\psi_x^{-1}(v)\in K_i\setminus K_{i-1}.
\]
\end{enumerate}
\end{itemize}
\end{enumerate}
\end{pro}

Before proving the proposition, we consider the case in which the
image of $\gamma$ is a partial submodule.

\begin{lem}\label{submod}
Let $\mathcal M$ be a partial submodule of $\mathcal N$. Denote by
\[
\gamma\colon\mathcal M\to\mathcal N
\]
the inclusion map and by
\[
\pi\colon\mathcal N\to\mathcal N/\mathcal M
\]
the quotient map. Then the sequence of right $R$-modules
\[
U(\mathcal M)
\xrightarrow{U(\gamma)}
U(\mathcal N)
\xrightarrow{U(\pi)}
U(\mathcal N/\mathcal M)
\longrightarrow 0
\]
is exact.
\end{lem}

\begin{rem}\label{equality}
We shall prove below that the sequence is also exact at
$U(\mathcal M)$; that is,
\[
0\longrightarrow U(\mathcal M)
\xrightarrow{U(\gamma)}
U(\mathcal N)
\]
is exact.
\end{rem}

\begin{proof}
The assertion follows directly from the construction of universal
enveloping modules in \eqref{constrenvelop}. The map $U(\pi)$ is
surjective because the image of $N$ generates
$U(\mathcal N/\mathcal M)$. Moreover, $\ker U(\pi)$ is the submodule
of $U(\mathcal N)$ generated by $M$, and therefore
\[
\ker U(\pi)=\operatorname{im}U(\gamma).
\]
\end{proof}

We are now ready to prove Proposition~\ref{isomorphism}.

\begin{proof}[Proof of Proposition~\ref{isomorphism}]
First assume that conditions \textup{(a)} and \textup{(b)} hold. By
the construction in \eqref{constrenvelop}, the passage from
$U(\mathcal N^{K_{i-1}})$ to $U(\mathcal N^{K_i})$ amounts to adding
one generator $u$ and a relation of the form
\[
u=v\cdot x^{\pm1},
\qquad x\in X,\quad v\in K_{i-1}.
\]
Since
\[
\chi(\mathcal N^{K_i})
=
\chi(\mathcal N^{K_{i-1}}),
\]
there are no additional independent relations. Consequently, for
each $i=1,\ldots,s$, the inclusion
\[
\mathcal N^{K_{i-1}}\to\mathcal N^{K_i}
\]
induces an isomorphism
\[
U(\mathcal N^{K_{i-1}})
\cong
U(\mathcal N^{K_i}).
\]
In particular,
\[
U(\mathcal N^{K_0})\cong U(\mathcal N),
\]
and hence
\[
\chi(\mathcal M)
=
\chi(\mathcal N)
=
\chi(\mathcal N^{K_0}).
\]

The map $\gamma$ induces a homomorphism
\[
\gamma'\colon\mathcal M\to\mathcal N^{K_0}.
\]
Since $\gamma$ is injective,
\[
m(\mathcal M)=m(\mathcal N^{K_0}),
\]
and, for every $x\in X$,
\[
m_x(\mathcal M)\leq m_x(\mathcal N^{K_0}).
\]
The equality of the Euler characteristics therefore implies
\[
m_x(\mathcal M)=m_x(\mathcal N^{K_0})
\qquad\text{for every }x\in X.
\]
Thus, $\gamma'$ is an isomorphism of partial $R$-modules, and it
follows that $U(\gamma)$ is an isomorphism.

Conversely, suppose that $U(\gamma)$ is an isomorphism. Since
\[
U(\gamma)\circ\alpha_{\mathcal M}
=
\alpha_{\mathcal N}\circ\gamma
\]
and the universal enveloping maps are injective on their base spaces,
$\gamma$ is injective. Moreover,
\[
\chi(\mathcal M)=\chi(\mathcal N).
\]
Thus, condition \textup{(a)} holds.

Set
\[
K_0=\gamma(M).
\]
Arguing as above, we see that the induced map
\[
\gamma'\colon\mathcal M\to\mathcal N^{K_0}
\]
is an isomorphism. We construct inductively a sequence
\[
\gamma(M)=K_0\subsetneq K_1\subsetneq\cdots\subsetneq K_s=N
\]
satisfying condition \textup{(b)}.

Suppose that $K_0,\ldots,K_t$ have been constructed in such a way
that, for every $1\leq i\leq t$,
\[
\dim_K(K_i/K_{i-1})=1
\]
and the inclusion
\[
\mathcal N^{K_{i-1}}\to\mathcal N^{K_i}
\]
induces an isomorphism of universal enveloping modules. Since
$U(\gamma)$ is an isomorphism, the inclusion
\[
\mathcal N^{K_t}\to\mathcal N
\]
also induces an isomorphism
\[
U(\mathcal N^{K_t})\cong U(\mathcal N).
\]

If $\mathcal N^{K_t}$ is a partial submodule of $\mathcal N$, then
Lemma~\ref{submod} implies that
\[
U(\mathcal N/\mathcal N^{K_t})=0.
\]
The canonical map from the base space of
$\mathcal N/\mathcal N^{K_t}$ to its universal enveloping module is
injective, so $N=K_t$, and the construction is complete.

Assume, therefore, that $\mathcal N^{K_t}$ is not a partial submodule
of $\mathcal N$. Then there exists $x\in X$ such that one of the
following holds:
\[
v\in K_t\cap W_x,
\qquad
u=\psi_x(v)\notin K_t,
\]
or
\[
v\in K_t\cap Z_x,
\qquad
u=\psi_x^{-1}(v)\notin K_t.
\]
Set
\[
K_{t+1}=K_t+Ku.
\]
The induced map
\[
U(\mathcal N^{K_t})
\longrightarrow
U(\mathcal N^{K_{t+1}})
\]
is injective because its composition with the map to
$U(\mathcal N)$ is an isomorphism. It is also surjective, since the
new vector $u$ is obtained from an element of $K_t$ by multiplication
by $x$ or $x^{-1}$. Hence it is an isomorphism, and therefore
\[
\chi(\mathcal N^{K_t})
=
\chi(\mathcal N^{K_{t+1}}).
\]
This completes the induction and proves condition \textup{(b)}.
\end{proof}

We note the following corollary to Lemma \ref{submod}, Remark \ref{equality}, and Proposition \ref{universalenveloping}.
\begin{cor}\label{submodule generated}
	Let $\mathcal{N}$ be a partial module and $\mathcal M$ the partial submodule, and $\iota:\mathcal{M}\hookrightarrow\mathcal{N}$ the inclusion map. Then $\alpha_{\mathcal{N}}^{-1}(\im(U(\iota)))=\M$ .
\end{cor}
\begin{proof}
	We have the following commutative diagram, where $\iota$ denotes the inclusion: 
	\begin{equation*}
		\begin{tikzcd}
		0 \arrow[r] & \mathcal{M} \arrow[r, "\iota"] \arrow[d, "\alpha_{\mathcal{M}}"] & \mathcal{N} \arrow[r, "q"] \arrow[d, "\alpha_{\mathcal{N}}"] & \mathcal{N}/\mathcal{M} \arrow[r] \arrow[d, "\alpha_{\mathcal{N}/\mathcal{M}}"] & 0 \\
		0 \arrow[r] & U(\mathcal{M}) \arrow[r, "U(\iota)"]                             & U(\mathcal{N}) \arrow[r, "U(q)"]                             & U(\mathcal{N}/\mathcal{M}) \arrow[r]                                            & 0
		\end{tikzcd}
	\end{equation*}
	Let $W=\alpha_{\mathcal{N}}^{-1}(\im(U(\iota)))$. Using the exactness of the second row first, and then commutativity of the rightmost square, we get
	\begin{equation*}
		W=\ker(U(q)\circ \alpha_{\mathcal{N}})=\ker(\alpha_{\mathcal{N}/\mathcal{M}}\circ q).
	\end{equation*}
	But $\alpha_{\mathcal{N}/\mathcal{M}}$ is injective, by Proposition \ref{universalenveloping}, so in fact
	\begin{equation*}
		W=\ker(q)=\mathcal{M},
	\end{equation*}   
	as needed. 
\end{proof}

\begin{rem}
  Note that in the particular case in which $\mathcal{M}$ is the partial submodule generated by some subset $S\subseteq\mathcal{N}$, the submodule $\im(U(\iota))$ is simply the submodule of $U(\mathcal{N})$ generated by $\alpha_{\mathcal{N}}(S)$. 
\end{rem}

We say that a partial $R$-module $\mathcal M$ is \textbf{free} if
$U(\mathcal M)$ is a free $R$-module, and that $\mathcal M$ is
\textbf{algebraic} if $U(\mathcal M)$ has no nontrivial free
quotients. We denote by $A(\mathcal M)$ the set of nontrivial
algebraic partial submodules of $\mathcal M$.

The preceding results imply that, when $K$ is finite, a finitely
presented $R$-module has only finitely many algebraic submodules.

\begin{pro}\label{freealgebraic}
Let
\[
\mathcal M
=
\bigl(M,(\phi_x\colon U_x\to V_x)_{x\in X}\bigr)
\]
be a finite-dimensional partial $R$-module, and let $L$ be a finitely
generated $R$-submodule of $U(\mathcal M)$.
\begin{enumerate}
\item[(a)] If
\[
L\cap M=\{0\},
\]
then $L$ is free.

\item[(b)] If $L$ is algebraic, then $L$ is generated as an
$R$-module by $L\cap M$.
\end{enumerate}
In particular, if $K$ is finite, then $A(L)$ is finite for every
finitely presented right $R$-module $L$.
\end{pro}

\begin{proof}
We first prove \textup{(a)}. Since $L$ is finitely generated, it is
finitely presented. By Proposition~\ref{presentationpartial}, there
exists a finite-dimensional partial $R$-module $\mathcal V$, with
base space $V$, and an isomorphism
\[
\tau\colon L\to U(\mathcal V).
\]

Apply Proposition~\ref{presentationpartial} to $U(\mathcal M)$ and to
a finite-dimensional subspace containing both $M$ and the image in
$U(\mathcal M)$ of $\tau^{-1}(V)$. We obtain a finite-dimensional
partial $R$-module
\[
\mathcal N
=
\bigl(N,(\psi_x\colon W_x\to Z_x)_{x\in X}\bigr)
\]
and a homomorphism
\[
\gamma\colon\mathcal M\to\mathcal N
\]
such that
\[
U(\gamma)\colon U(\mathcal M)\to U(\mathcal N)
\]
is an isomorphism and $N$ contains
\[
V'
=
U(\gamma)\bigl(\tau^{-1}(V)\bigr).
\]
The induced partial module $\mathcal N^{V'}$ is isomorphic to
$\mathcal V$. Moreover,
\[
V'\cap\gamma(M)=\{0\},
\]
because $L\cap M=\{0\}$.

By Proposition~\ref{isomorphism}, there exists a sequence
\[
\gamma(M)=K_0\subsetneq K_1\subsetneq\cdots\subsetneq K_s=N
\]
such that, for every $i=1,\ldots,s$,
\[
\dim_K(K_i/K_{i-1})=1,
\qquad
\chi(\mathcal N^{K_i})
=
\chi(\mathcal N^{K_{i-1}}),
\]
and there exists $x\in X$ for which the new vector in
$K_i\setminus K_{i-1}$ is obtained from an element of $K_{i-1}$ by
applying $\psi_x$ or $\psi_x^{-1}$.

Set
\[
L_i=K_i\cap V'.
\]
We prove by induction on $i$ that $\mathcal N^{L_i}$ is free. Since
\[
L_0=V'\cap\gamma(M)=\{0\},
\]
the assertion holds for $i=0$. Suppose that
$\mathcal N^{L_{i-1}}$ is free. If $L_i=L_{i-1}$, there is nothing to
prove. We may therefore assume that
\[
\dim_K(L_i/L_{i-1})=1.
\]

Let $x\in X$ be the generator associated with the passage from
$K_{i-1}$ to $K_i$. As observed in the proof of
Proposition~\ref{isomorphism},
\[
m_y(\mathcal N^{K_i})
=
m_y(\mathcal N^{K_{i-1}})
\qquad (y\in X,\ y\neq x),
\]
whereas
\[
m_x(\mathcal N^{K_i})
=
m_x(\mathcal N^{K_{i-1}})+1.
\]
It follows that
\[
m_y(\mathcal N^{L_i})
=
m_y(\mathcal N^{L_{i-1}})
\qquad (y\in X,\ y\neq x),
\]
and
\[
m_x(\mathcal N^{L_i})
\leq
m_x(\mathcal N^{L_{i-1}})+1.
\]

If
\[
m_x(\mathcal N^{L_i})
=
m_x(\mathcal N^{L_{i-1}}),
\]
then the new basis vector introduces no new relation, and hence
\[
U(\mathcal N^{L_i})
\cong
U(\mathcal N^{L_{i-1}})\oplus R.
\]
If
\[
m_x(\mathcal N^{L_i})
=
m_x(\mathcal N^{L_{i-1}})+1,
\]
then the new basis vector is accompanied by one relation, and hence
\[
U(\mathcal N^{L_i})
\cong
U(\mathcal N^{L_{i-1}}).
\]
In either case, $U(\mathcal N^{L_i})$ is free. This completes the
induction.

Since $K_s=N$, we have $L_s=V'$, and therefore
$\mathcal N^{V'}$ is free. Thus, $\mathcal V$ is free, and hence $L$
is free.

We now prove \textup{(b)}. Let $L'$ be the $R$-submodule of
$U(\mathcal M)$ generated by $L\cap M$, and put
\[
N=L'\cap M.
\]
Then $\mathcal N=\mathcal M^N$ is a partial submodule of $\mathcal M$.
Let
\[
\pi\colon\mathcal M\to\mathcal M/\mathcal N
\]
be the quotient map. By Lemma~\ref{submod},
\[
\ker U(\pi)
\]
is the submodule of $U(\mathcal M)$ generated by $N$, and hence
\[
\ker U(\pi)=L'.
\]
Furthermore,
\[
U(\pi)(L)\cap(M/N)=\{0\}.
\]
Indeed, if the image of an element of $L$ lies in $M/N$, then,
modulo $L'$, that element is represented by an element of $L\cap M$,
and hence its image is zero.

Part \textup{(a)} now implies that $U(\pi)(L)$ is free. Since it is a
quotient of the algebraic module $L$, it must be trivial. Therefore,
\[
L=L',
\]
as required.

Finally, suppose that $K$ is finite and that $L$ is finitely
presented. By Proposition~\ref{presentationpartial}, write
\[
L\cong U(\mathcal M)
\]
for a finite-dimensional partial module $\mathcal M$ with base space
$M$. By part \textup{(b)}, every algebraic submodule of $L$ is
generated by its intersection with $M$. Since a finite-dimensional
vector space over a finite field has only finitely many subspaces,
$A(L)$ is finite.
\end{proof}

We can now prove Remark~\ref{equality}.

\begin{cor}\label{submod2}
Let $\mathcal M$ be a finite-dimensional partial submodule of a
finite-dimensional partial $R$-module $\mathcal N$. Denote by
\[
\gamma\colon\mathcal M\to\mathcal N
\]
the inclusion map. Then
\[
U(\gamma)\colon U(\mathcal M)\to U(\mathcal N)
\]
is injective.
\end{cor}

\begin{proof}
Let
\[
\pi\colon\mathcal N\to\mathcal N/\mathcal M
\]
be the quotient map. By Lemma~\ref{submod}, the sequence
\[
U(\mathcal M)
\xrightarrow{U(\gamma)}
U(\mathcal N)
\xrightarrow{U(\pi)}
U(\mathcal N/\mathcal M)
\longrightarrow 0
\]
is exact. Therefore,
\begin{align*}
\chi\bigl(\ker U(\gamma)\bigr)
&=
\chi\bigl(U(\mathcal M)\bigr)
+
\chi\bigl(U(\mathcal N/\mathcal M)\bigr)
-
\chi\bigl(U(\mathcal N)\bigr)\\
&=
\chi(\mathcal M)
+
\chi(\mathcal N/\mathcal M)
-
\chi(\mathcal N)\\
&=0.
\end{align*}

Moreover,
\[
\ker U(\gamma)\cap M=\{0\},
\]
because the universal enveloping maps are injective on their base
spaces. Proposition~\ref{freealgebraic}\textup{(a)} therefore implies
that $\ker U(\gamma)$ is free. A finitely generated free
$R$-module of Euler characteristic zero is trivial. Hence,
\[
\ker U(\gamma)=\{0\}.
\]
\end{proof}
\section{Counting homomorphisms to a random $R$-module of dimension $n$}

We retain the notation of the previous section. Throughout this section,
we assume that $K=\mathbb F_q$ is the finite field with $q$ elements.
If $V$ is a $K$-vector space, we write $d(V)$ for its dimension over
$K$. Let $F$ be the free group freely generated by a finite set $X$ of
cardinality $d$, and put $R=K[F]$.

If
\[
\phi\colon F\longrightarrow \operatorname{GL}_n(q)
\]
is a homomorphism, we denote the induced right $R$-module by
$(\mathbb F_q^n)_\phi$. For a finite-dimensional partial $R$-module
$\mathcal M$, define
\[
\Lambda_{\mathcal M}(n)
=
\frac{1}{|\operatorname{GL}_n(q)|^d}
\sum_{\phi\colon F\to\operatorname{GL}_n(q)}
\left|
\operatorname{Hom}_R
\bigl(\mathcal M,(\mathbb F_q^n)_\phi\bigr)
\right|.
\]
Thus, $\Lambda_{\mathcal M}(n)$ is the average number of
$R$-homomorphisms from $\mathcal M$ to a random $R$-module whose
underlying $K$-vector space has dimension $n$. By the universal
property of the enveloping module,
\[
\Lambda_{\mathcal M}(n)=\Lambda_{U(\mathcal M)}(n).
\]

From now on, we fix a finite-dimensional partial $R$-module
$\mathcal M$ with base space $M$. Let $P=P_{\mathcal M}$ be the poset
of partial $R$-submodules of $\mathcal M$. Since $K$ is finite and $M$
is finite-dimensional, the poset $P$ is finite.

We shall use the incidence algebra $I(P)$ of integer-valued functions
on
\[
\{(\mathcal N_1,\mathcal N_2)\in P\times P
  \mid \mathcal N_1\leq \mathcal N_2\}.
\]
In addition to pointwise addition and scalar multiplication, $I(P)$
has an associative multiplication given by convolution:
\[
(f*g)(\mathcal N_1,\mathcal N_2)
=
\sum_{\mathcal N_1\leq\mathcal N\leq\mathcal N_2}
f(\mathcal N_1,\mathcal N)
g(\mathcal N,\mathcal N_2).
\]
The identity element is the diagonal function
\[
\delta(\mathcal N_1,\mathcal N_2)
=
\begin{cases}
1,&\mathcal N_1=\mathcal N_2,\\
0,&\mathcal N_1<\mathcal N_2.
\end{cases}
\]
Every element of $I(P)$ whose diagonal entries are invertible is
invertible with respect to convolution. In particular, the zeta
function
\[
\zeta(\mathcal N_1,\mathcal N_2)=1
\qquad
(\mathcal N_1\leq\mathcal N_2)
\]
is invertible. Its inverse, denoted by $\mu$, is the M\"obius function
of $P$.

For $\mathcal N_1\leq\mathcal N_2\leq\mathcal M$, set
\[
\Phi_n(\mathcal N_1,\mathcal N_2)
=
q^{-\chi(\mathcal N_2)n}
\Lambda_{\mathcal N_2/\mathcal N_1}(n).
\]
The normalizing factor is chosen so that
\[
\Phi_n(\mathcal N_1,\mathcal N_2)
=
q^{-\chi(\mathcal N_1)n}
\]
whenever $\mathcal N_2/\mathcal N_1$ is free. We study $\Phi_n$ using
the approach of \cite{PP15}. For each $n$, the function $\Phi_n$ is an
element of $I(P)$.

Define
\[
R_n=\Phi_n*\mu,
\qquad
L_n=\mu*\Phi_n,
\qquad
C_n=\mu*\Phi_n*\mu.
\]
Then
\begin{align}
\label{derivation phi}
\Phi_n(\mathcal N_1,\mathcal N_2)
&=
\sum_{\mathcal N_1\leq\mathcal N\leq\mathcal N_2}
R_n(\mathcal N_1,\mathcal N)
=
\sum_{\mathcal N_1\leq\mathcal N\leq\mathcal N_2}
L_n(\mathcal N,\mathcal N_2),
\\
R_n(\mathcal N_1,\mathcal N_2)
&=
\sum_{\mathcal N_1\leq\mathcal N\leq\mathcal N_2}
C_n(\mathcal N_1,\mathcal N),
\nonumber\\
L_n(\mathcal N_1,\mathcal N_2)
&=
\sum_{\mathcal N_1\leq\mathcal N\leq\mathcal N_2}
C_n(\mathcal N,\mathcal N_2).
\nonumber
\end{align}

\begin{pro}\label{rightnonalgebraic}
Let
\[
\mathcal N_1\leq\mathcal N_2\leq\mathcal N_3\leq\mathcal M
\]
be a chain of partial $R$-submodules. If
$\mathcal N_3/\mathcal N_2$ is free, then
\[
\Phi_n(\mathcal N_1,\mathcal N_2)
=
\Phi_n(\mathcal N_1,\mathcal N_3).
\]
In particular,
\[
R_n(\mathcal N_1,\mathcal N_2)=0
\]
if $\mathcal N_2/\mathcal N_1$ is not algebraic.
\end{pro}

\begin{proof}
Let
\[
\gamma\colon
\mathcal N_2/\mathcal N_1
\longrightarrow
\mathcal N_3/\mathcal N_1
\]
be the inclusion, and let
\[
\pi\colon
\mathcal N_3/\mathcal N_1
\longrightarrow
\mathcal N_3/\mathcal N_2
\]
be the quotient map. By Lemma~\ref{submod} and
Corollary~\ref{submod2}, the sequence
\[
0
\longrightarrow U(\mathcal N_2/\mathcal N_1)
\xrightarrow{U(\gamma)}
U(\mathcal N_3/\mathcal N_1)
\xrightarrow{U(\pi)}
U(\mathcal N_3/\mathcal N_2)
\longrightarrow 0
\]
is exact. If $\mathcal N_3/\mathcal N_2$ is free, then this sequence
splits, and hence
\[
U(\mathcal N_3/\mathcal N_1)
\cong
U(\mathcal N_2/\mathcal N_1)
\oplus
U(\mathcal N_3/\mathcal N_2).
\]
Consequently,
\[
\Lambda_{\mathcal N_3/\mathcal N_1}(n)
=
q^{\chi(\mathcal N_3/\mathcal N_2)n}
\Lambda_{\mathcal N_2/\mathcal N_1}(n).
\]
It follows that
\begin{align*}
\Phi_n(\mathcal N_1,\mathcal N_3)
&=
q^{-\chi(\mathcal N_3)n}
\Lambda_{\mathcal N_3/\mathcal N_1}(n)\\
&=
q^{-\chi(\mathcal N_3)n}
q^{\chi(\mathcal N_3/\mathcal N_2)n}
\Lambda_{\mathcal N_2/\mathcal N_1}(n)\\
&=
q^{-\chi(\mathcal N_2)n}
\Lambda_{\mathcal N_2/\mathcal N_1}(n)\\
&=
\Phi_n(\mathcal N_1,\mathcal N_2).
\end{align*}
The final assertion follows from the same M\"obius-inversion argument
as in \cite[Proposition~5.1 and Remark~5.2]{PP15}.
\end{proof}

For an $R$-module $V$, let
\[
\operatorname{Hom}^{\mathrm{inj}}_R(\mathcal M,V)
\]
denote the set of injective homomorphisms from $\mathcal M$ to $V$.
For integers $n\geq l\geq 1$, define
\[
T_{n,l}(q)
=
(1-q^{-n})(1-q^{-n+1})\cdots(1-q^{-n+l-1}),
\]
and put $T_{n,0}(q)=1$.

\begin{rem}\label{o: meaning of T}
Let $V$ and $W$ be finite-dimensional $K$-vector spaces with
$d(V)\leq d(W)$. The number of injective $K$-linear maps from $V$ to
$W$ is
\[
q^{d(W)d(V)}T_{d(W),d(V)}(q).
\]

Similarly, suppose that $V,V'\leq W$ and that
$\psi\colon V\to V'$ is a $K$-linear isomorphism. The number of
automorphisms of $W$ extending $\psi$ is
\[
\frac{|\operatorname{GL}_{d(W)}(q)|}
{q^{d(W)d(V)}T_{d(W),d(V)}(q)}.
\]
\end{rem}

\begin{pro}\label{p: first formula for L}
Let
\[
\mathcal N_1\leq\mathcal N_2\leq\mathcal M
\]
be a chain of partial $R$-submodules. Then
\begin{align*}
L_n(\mathcal N_1,\mathcal N_2)
&=
\frac{q^{-\chi(\mathcal N_2)n}}
{|\operatorname{GL}_n(q)|^d}
\sum_{\phi\colon F\to\operatorname{GL}_n(q)}
\left|
\operatorname{Hom}^{\mathrm{inj}}_R
\bigl(
\mathcal N_2/\mathcal N_1,
(\mathbb F_q^n)_\phi
\bigr)
\right|\\
&=
q^{-\chi(\mathcal N_1)n}
\frac{
T_{n,m(\mathcal N_2/\mathcal N_1)}(q)
}{
\displaystyle
\prod_{x\in X}
T_{n,m_x(\mathcal N_2/\mathcal N_1)}(q)
}.
\end{align*}
\end{pro}

\begin{proof}
For
$\mathcal N_1\leq\mathcal N_2\leq\mathcal M$, define
\[
\widehat L_n(\mathcal N_1,\mathcal N_2)
=
\frac{q^{-\chi(\mathcal N_2)n}}
{|\operatorname{GL}_n(q)|^d}
\sum_{\phi\colon F\to\operatorname{GL}_n(q)}
\left|
\operatorname{Hom}^{\mathrm{inj}}_R
\bigl(
\mathcal N_2/\mathcal N_1,
(\mathbb F_q^n)_\phi
\bigr)
\right|.
\]
To prove the first equality, it is enough to verify that
\[
\Phi_n(\mathcal N_1,\mathcal N_2)
=
\sum_{\mathcal N_1\leq\mathcal N\leq\mathcal N_2}
\widehat L_n(\mathcal N,\mathcal N_2).
\]
For a fixed homomorphism
$\phi\colon F\to\operatorname{GL}_n(q)$, every homomorphism
\[
f\colon
\mathcal N_2/\mathcal N_1
\longrightarrow
(\mathbb F_q^n)_\phi
\]
has a unique kernel of the form
$\mathcal N/\mathcal N_1$, where
$\mathcal N_1\leq\mathcal N\leq\mathcal N_2$, and it induces an
injective homomorphism
\[
\overline f\colon
\mathcal N_2/\mathcal N
\longrightarrow
(\mathbb F_q^n)_\phi.
\]
Thus,
\[
\left|
\operatorname{Hom}_R
\bigl(
\mathcal N_2/\mathcal N_1,
(\mathbb F_q^n)_\phi
\bigr)
\right|
=
\sum_{\mathcal N_1\leq\mathcal N\leq\mathcal N_2}
\left|
\operatorname{Hom}^{\mathrm{inj}}_R
\bigl(
\mathcal N_2/\mathcal N,
(\mathbb F_q^n)_\phi
\bigr)
\right|,
\]
which proves the first equality by M\"obius inversion.

For the second equality, let $\mathcal P$ be the set of pairs
$(\phi,\iota)$, where
\[
\phi\colon F\to\operatorname{GL}_n(q)
\]
is a homomorphism and
\[
\iota\in
\operatorname{Hom}^{\mathrm{inj}}_R
\bigl(
\mathcal N_2/\mathcal N_1,
(\mathbb F_q^n)_\phi
\bigr).
\]
Let $N$ be the base space of $\mathcal N_2/\mathcal N_1$. For
$n\geq m(\mathcal N_2/\mathcal N_1)$, the number of injective
$K$-linear maps
\[
\iota\colon N\longrightarrow\mathbb F_q^n
\]
is
\[
q^{n\,m(\mathcal N_2/\mathcal N_1)}
T_{n,m(\mathcal N_2/\mathcal N_1)}(q).
\]
Once $\iota$ is fixed, for each $x\in X$ the partial action of $x$
prescribes an isomorphism between two subspaces of $\mathbb F_q^n$ of
dimension $m_x(\mathcal N_2/\mathcal N_1)$. By
Remark~\ref{o: meaning of T}, the number of possible extensions to an
element of $\operatorname{GL}_n(q)$ is
\[
\frac{|\operatorname{GL}_n(q)|}
{q^{n\,m_x(\mathcal N_2/\mathcal N_1)}
T_{n,m_x(\mathcal N_2/\mathcal N_1)}(q)}.
\]
Therefore,
\begin{align*}
|\mathcal P|
&=
\frac{
q^{n\,m(\mathcal N_2/\mathcal N_1)}
T_{n,m(\mathcal N_2/\mathcal N_1)}(q)
|\operatorname{GL}_n(q)|^d
}{
\displaystyle
\prod_{x\in X}
q^{n\,m_x(\mathcal N_2/\mathcal N_1)}
T_{n,m_x(\mathcal N_2/\mathcal N_1)}(q)
}\\
&=
q^{\chi(\mathcal N_2/\mathcal N_1)n}
\frac{
T_{n,m(\mathcal N_2/\mathcal N_1)}(q)
|\operatorname{GL}_n(q)|^d
}{
\displaystyle
\prod_{x\in X}
T_{n,m_x(\mathcal N_2/\mathcal N_1)}(q)
}.
\end{align*}
Since
\[
-\chi(\mathcal N_2)
+\chi(\mathcal N_2/\mathcal N_1)
=
-\chi(\mathcal N_1),
\]
the desired formula follows.
\end{proof}

\begin{exam}
Let $X=\{x_1,\ldots,x_d\}$ and let $V=K$. For
$i=1,\ldots,t$, put
\[
U_{x_i}=V_{x_i}=V,
\]
and for $i=t+1,\ldots,d$, put
\[
U_{x_i}=V_{x_i}=\{0\}.
\]
Let every nonzero partial action $\phi_{x_i}$ be the identity, and set
\[
\mathcal V=(V,\phi_x\colon U_x\to V_x)_{x\in X}.
\]
If $H=\langle x_1,\ldots,x_t\rangle$, then
\[
U(\mathcal V)\cong K[H\backslash F].
\]
Moreover,
\begin{align*}
\Lambda_{K[H\backslash F]}(n)
&=
\Lambda_{\mathcal V}(n)\\
&=
q^{\chi(\mathcal V)n}
\Phi_n(\{0\},\mathcal V)\\
&=
q^{\chi(\mathcal V)n}
\bigl(
L_n(\mathcal V,\mathcal V)
+
L_n(\{0\},\mathcal V)
\bigr)\\
&=
1+\frac{1}{(q^n-1)^{t-1}}.
\end{align*}
\end{exam}

By a \emph{filtration} $\mathcal F$ of a finite-dimensional
$K$-vector space $U$, we mean a chain
\[
\{0\}=U_0<U_1<\cdots<U_{d(U)}=U
\]
such that
\[
\dim_K(U_j/U_{j-1})=1
\]
for every $1\leq j\leq d(U)$.

Let
\[
\mathcal F=(\{0\}=U_0<U_1<\cdots<U_{d(U)}=U)
\]
be a filtration of $U$, and let $W\leq V\leq U$. The filtration
\emph{induced} by $\mathcal F$ on $V/W$, denoted by
$\mathcal F_{\restriction V/W}$, is obtained by removing repetitions
from the chain
\[
W/W
\leq
\frac{(U_1\cap V)+W}{W}
\leq
\cdots
\leq
\frac{(U_{d(U)}\cap V)+W}{W}
=
V/W.
\]

Let
\[
\mathcal F=(\{0\}=V_0<V_1<\cdots<V_{d(V)}=V)
\]
be a filtration of a finite-dimensional $K$-vector space $V$. For
$k=1,\ldots,d(V)$, put
\[
Q_k^{\mathcal F}
=
\{Kv\in\mathbb P(V)\mid v\in V_k\setminus V_{k-1}\}.
\]
Then
\[
|Q_k^{\mathcal F}|
=
\frac{|V_k|-|V_{k-1}|}{q-1}
=
q^{d(V_{k-1})}
=
q^{k-1}.
\]
We also write
\[
\overline Q^{\mathcal F}
=
(Q_k^{\mathcal F})_{1\leq k\leq d(V)}.
\]

Let $\mathcal X=(X_i)_{i\in I}$ be a family of sets. Denote by
$\mathcal S(\mathcal X)$ the collection of finite subsets $S$ of the
disjoint union
\[
\bigsqcup_{i\in I}X_i
\]
such that
\[
|S\cap X_i|\leq 1
\]
for every $i\in I$. Even when the sets $X_i$ intersect in an ambient
space, we regard them as disjoint copies in the coproduct. For
$l\geq0$, let $\mathcal S_l(\mathcal X)$ be the collection of elements
of $\mathcal S(\mathcal X)$ of cardinality $l$.

\begin{lem}\label{l: expanding the Ts}
Let $U$ be a finite-dimensional $K$-vector space, let
$n\geq d(U)$, and let $\mathcal F$ be a filtration of $U$. Then
\begin{align}
\label{l: calculating T}
T_{n,d(U)}(q)
&=
\sum_{j=0}^{d(U)}
(-1)^j q^{-nj}
\left|
\mathcal S_j(\overline Q^{\mathcal F})
\right|\\
&=
\sum_{S\in\mathcal S(\overline Q^{\mathcal F})}
(-1)^{|S|}q^{-n|S|}.
\nonumber
\end{align}
More generally, let $U_1,\ldots,U_d$ be finite-dimensional
$K$-vector spaces, let
\[
n\geq\max_{1\leq i\leq d}d(U_i),
\]
and let $\mathcal F_i$ be a filtration of $U_i$. Then
\begin{align*}
\prod_{i=1}^d T_{n,d(U_i)}(q)
&=
\sum_j
(-1)^j q^{-nj}
\left|
\mathcal S_j
\bigl(
(\overline Q^{\mathcal F_i})_{1\leq i\leq d}
\bigr)
\right|\\
&=
\sum_{
S\in
\mathcal S((\overline Q^{\mathcal F_i})_{1\leq i\leq d})
}
(-1)^{|S|}q^{-n|S|}.
\end{align*}
\end{lem}

\begin{proof}
Let $\mathcal X=(X_i)_{1\leq i\leq l}$ be a family of finite sets and
write $|X_i|=a_i$. In $\mathbb Z[s]$ we have
\[
(1-a_1s)\cdots(1-a_ls)
=
\sum_{j=0}^l(-1)^j c_js^j,
\]
where
\[
c_j=|\mathcal S_j(\mathcal X)|.
\]
Taking $\mathcal X=\overline Q^{\mathcal F}$ and $s=q^{-n}$ proves
the first assertion. The second is proved in exactly the same way.
\end{proof}

For a set $A$, write
\[
A^{<\omega}=\bigcup_{l\geq0}A^l
\]
for the set of finite tuples with entries in $A$. The set $A^0$
contains a unique element, the empty tuple.

The following standard generating-function identity will be used
below.

\begin{lem}\label{invertingL}
Let $\mathcal T$ be a collection of subsets of a finite set and assume
that $\emptyset\in\mathcal T$. For a tuple
\[
T=(S_1,\ldots,S_l)
\in
(\mathcal T\setminus\{\emptyset\})^{<\omega},
\]
put
\[
|T|=\sum_{i=1}^l|S_i|
\qquad\text{and}\qquad
l(T)=l.
\]
For the empty tuple, put $|T|=l(T)=0$. Then
\[
\left(
\sum_{S\in\mathcal T}s^{|S|}
\right)^{-1}
=
\sum_{
T\in(\mathcal T\setminus\{\emptyset\})^{<\omega}
}
(-1)^{l(T)}s^{|T|}
\]
in $\mathbb Z[[s]]$.
\end{lem}

Let
\[
\mathcal L=(L,\phi_x\colon W_x\to Z_x)_{x\in X}
\]
be a finite-dimensional partial $R$-module, and let $\mathcal F$ be a
filtration of its base space $L$. Define
\[
\mathbb T_{\mathcal L}^{\mathcal F}
=
\mathcal S(\overline Q^{\mathcal F})
\times
\left(
\mathcal S
\bigl(
(\overline Q^{\mathcal F_{\restriction W_x}})_{x\in X}
\bigr)
\setminus\{\emptyset\}
\right)^{<\omega}.
\]
If
\[
T=(S,(S_j)_{j=1}^l)
\in
\mathbb T_{\mathcal L}^{\mathcal F},
\]
set
\[
|T|=|S|+\sum_{j=1}^l|S_j|
\qquad\text{and}\qquad
\epsilon(T)=(-1)^{|T|+l}.
\]
We denote by $\langle T\rangle$ the partial $R$-submodule of
$\mathcal L$ generated by the one-dimensional subspaces represented
by the elements of
\[
S\cup\bigcup_{j=1}^l S_j.
\]

\begin{pro}\label{L formula}
Let
\[
\mathcal N_1\leq\mathcal N_2\leq\mathcal M
\]
be a chain of partial $R$-submodules, and put
\[
\mathcal L=\mathcal N_2/\mathcal N_1.
\]
Let $\mathcal F$ be a filtration of the base space of $\mathcal L$.
Then
\[
L_n(\mathcal N_1,\mathcal N_2)
=
q^{-\chi(\mathcal N_1)n}
\sum_{T\in\mathbb T_{\mathcal L}^{\mathcal F}}
\epsilon(T)q^{-|T|n}.
\]
\end{pro}

\begin{proof}
By Proposition~\ref{p: first formula for L} and
Lemma~\ref{l: expanding the Ts},
\[
q^{\chi(\mathcal N_1)n}
L_n(\mathcal N_1,\mathcal N_2)
=
\frac{
\displaystyle
\sum_{S\in\mathcal S(\overline Q^{\mathcal F})}
(-1)^{|S|}q^{-n|S|}
}{
\displaystyle
\sum_{
S\in
\mathcal S
((\overline Q^{\mathcal F_{\restriction W_x}})_{x\in X})
}
(-1)^{|S|}q^{-n|S|}
}.
\]
Applying Lemma~\ref{invertingL} with $s=-q^{-n}$ gives
\begin{align*}
&
\left(
\sum_{
S\in
\mathcal S
((\overline Q^{\mathcal F_{\restriction W_x}})_{x\in X})
}
(-1)^{|S|}q^{-n|S|}
\right)^{-1}
\\
&\qquad=
\sum_{
T\in
\left(
\mathcal S
((\overline Q^{\mathcal F_{\restriction W_x}})_{x\in X})
\setminus\{\emptyset\}
\right)^{<\omega}
}
(-1)^{|T|+l(T)}q^{-n|T|}.
\end{align*}
Multiplying the two series yields the desired formula.
\end{proof}

The following simple observation will be useful.

\begin{lem}\label{o: restriction filtration}
Let $V\leq U$ be finite-dimensional $K$-vector spaces, and let
$\mathcal F$ be a filtration of $U$. Then
\[
\mathcal S(\overline Q^{\mathcal F})
\cap
\mathcal P(\mathbb P(V))
=
\mathcal S(\overline Q^{\mathcal F_{\restriction V}}),
\]
where the left-hand side consists of those elements of
$\mathcal S(\overline Q^{\mathcal F})$ all of whose projective points
belong to $\mathbb P(V)$.
\end{lem}

\begin{proof}
Write
\[
\mathcal F=(\{0\}=U_0<U_1<\cdots<U_{d(U)}=U).
\]
After repetitions are removed, the induced filtration of $V$ is
\[
\{0\}=V_0<V_1<\cdots<V_{d(V)}=V.
\]
For each $j$, there is a unique index $i_j$ at which the dimension of
$U_i\cap V$ increases from $j-1$ to $j$. Consequently,
\[
Q_j^{\mathcal F_{\restriction V}}
=
Q_{i_j}^{\mathcal F}\cap\mathbb P(V).
\]
Thus, a subset of $\mathbb P(V)$ contains at most one point from each
layer of the ambient filtration if and only if it contains at most one
point from each layer of the induced filtration.
\end{proof}

\begin{pro}\label{formula for C}
Let
\[
\mathcal N_1\leq\mathcal N_2\leq\mathcal M
\]
be partial $R$-submodules, and let $\mathcal F$ be a filtration of the
base space $M$ of $\mathcal M$. Put
\[
\mathcal L=\mathcal N_2/\mathcal N_1.
\]
Then
\[
C_n(\mathcal N_1,\mathcal N_2)
=
q^{-\chi(\mathcal N_1)n}
\sum_{\substack{
T\in
\mathbb T_{\mathcal M/\mathcal N_1}^{
\mathcal F_{\restriction M/N_1}
}\\
\langle T\rangle=\mathcal L
}}
\epsilon(T)q^{-|T|n}.
\]
In particular,
\[
C_n(\mathcal N_1,\mathcal N_2)
=
O\left(
q^{-n(\chi(\mathcal N_1)+\operatorname{rk}(\mathcal L))}
\right).
\]
If $\mathcal L$ is nontrivial and algebraic, then
\[
C_n(\mathcal N_1,\mathcal N_2)
=
O\left(q^{-n(\chi(\mathcal N_2)+1)}\right).
\]
\end{pro}

\begin{proof}
Let
\[
\mathcal N_1\leq\mathcal N_2\leq\mathcal N_3\leq\mathcal M,
\]
and denote their base spaces by
$N_1\leq N_2\leq N_3$, respectively. Define
\[
\widetilde C_n
(\mathcal N_1,\mathcal N_2,\mathcal N_3)
=
q^{-\chi(\mathcal N_1)n}
\sum_{\substack{
T\in
\mathbb T_{\mathcal N_3/\mathcal N_1}^{
\mathcal F_{\restriction N_3/N_1}
}\\
\langle T\rangle=\mathcal N_2/\mathcal N_1
}}
\epsilon(T)q^{-|T|n}.
\]
By Lemma~\ref{o: restriction filtration},
\begin{align*}
&
\left\{
T\in
\mathbb T_{\mathcal N_3/\mathcal N_1}^{
\mathcal F_{\restriction N_3/N_1}
}
\mid
\langle T\rangle=\mathcal N_2/\mathcal N_1
\right\}
\\
&\qquad=
\left\{
T\in
\mathbb T_{\mathcal N_2/\mathcal N_1}^{
\mathcal F_{\restriction N_2/N_1}
}
\mid
\langle T\rangle=\mathcal N_2/\mathcal N_1
\right\}.
\end{align*}
Hence
\[
\widetilde C_n
(\mathcal N_1,\mathcal N_2,\mathcal N_3)
=
\widetilde C_n
(\mathcal N_1,\mathcal N_2,\mathcal N_2),
\]
and we denote this common value simply by
$\widetilde C_n(\mathcal N_1,\mathcal N_2)$.

By Proposition~\ref{L formula}, for all sufficiently large $n$,
\[
L_n(\mathcal N_1,\mathcal N_3)
=
\sum_{\mathcal N_1\leq\mathcal N_2\leq\mathcal N_3}
\widetilde C_n
(\mathcal N_1,\mathcal N_2,\mathcal N_3)
=
\sum_{\mathcal N_1\leq\mathcal N_2\leq\mathcal N_3}
\widetilde C_n(\mathcal N_1,\mathcal N_2).
\]
Comparing this identity with \eqref{derivation phi} gives
\[
C_n(\mathcal N_1,\mathcal N_2)
=
\widetilde C_n(\mathcal N_1,\mathcal N_2).
\]

Every $T$ with $\langle T\rangle=\mathcal L$ satisfies
\[
|T|\geq\operatorname{rk}(\mathcal L),
\]
which proves the first estimate. If $\mathcal L$ is nontrivial and
algebraic, then it cannot be generated by $\chi(\mathcal L)$ or fewer
elements. Hence
\[
|T|\geq\chi(\mathcal L)+1.
\]
Since
\[
\chi(\mathcal N_2)
=
\chi(\mathcal N_1)+\chi(\mathcal L),
\]
the second estimate follows.
\end{proof}

The following is an immediate consequence.

\begin{cor}\label{oright}
Let
\[
\mathcal N_1\leq\mathcal N_2\leq\mathcal M
\]
be a chain of partial $R$-submodules, and assume that
$\mathcal N_2/\mathcal N_1$ is algebraic. Then
\[
R_n(\mathcal N_1,\mathcal N_2)
=
q^{-\chi(\mathcal N_2)n}
+
O\left(q^{-n(\chi(\mathcal N_2)+1)}\right).
\]
\end{cor}

\begin{proof}
By \eqref{derivation phi},
\[
R_n(\mathcal N_1,\mathcal N_2)
=
\sum_{\mathcal N_1\leq\mathcal N\leq\mathcal N_2}
C_n(\mathcal N,\mathcal N_2).
\]
Proposition~\ref{formula for C} gives
\[
C_n(\mathcal N_2,\mathcal N_2)
=
q^{-\chi(\mathcal N_2)n}.
\]
For
$\mathcal N_1\leq\mathcal N<\mathcal N_2$, the quotient
$\mathcal N_2/\mathcal N$ is nontrivial and algebraic, because it is a
quotient of the algebraic partial module
$\mathcal N_2/\mathcal N_1$. Therefore,
\[
C_n(\mathcal N,\mathcal N_2)
=
O\left(q^{-n(\chi(\mathcal N_2)+1)}\right),
\]
and the result follows.
\end{proof}

We extend the notation introduced in the Introduction by defining
\[
f(\mathcal M)
=
\begin{cases}
+\infty,
&\text{if $\mathcal M$ is free},\\[2mm]
\displaystyle
\min\{\chi(\mathcal N)\mid\mathcal N\in A(\mathcal M)\},
&\text{otherwise},
\end{cases}
\]
and
\[
\operatorname{Crit}(\mathcal M)
=
\{
\mathcal N\in A(\mathcal M)
\mid
\chi(\mathcal N)=f(\mathcal M)
\}.
\]

\begin{cor}\label{uM}
Let
\[
\mathcal N_1\leq\mathcal N_2\leq\mathcal M
\]
be a chain of partial $R$-submodules, and put
\[
\mathcal L=\mathcal N_2/\mathcal N_1.
\]
Then
\begin{align}
\Phi_n(\mathcal N_1,\mathcal N_2)
&=
q^{-\chi(\mathcal N_1)n}
+
\sum_{\mathcal N\in A(\mathcal L)}
q^{-(\chi(\mathcal N_1)+\chi(\mathcal N))n}
\bigl(1+O(q^{-n})\bigr)
\notag\\
&=
q^{-\chi(\mathcal N_1)n}
+
|\operatorname{Crit}(\mathcal L)|
q^{-(\chi(\mathcal N_1)+f(\mathcal L))n}
\bigl(1+O(q^{-n})\bigr).
\label{phi-asymptotic}
\end{align}

In particular,
\begin{align}
\label{lambda-partial-asymptotic}
\Lambda_{\mathcal M}(n)
&=
q^{\chi(\mathcal M)n}
+
\sum_{\mathcal N\in A(\mathcal M)}
q^{\chi(\mathcal M/\mathcal N)n}
\bigl(1+O(q^{-n})\bigr)\notag \\ &=
q^{\chi(\mathcal M)n}
+
|\operatorname{Crit}(\mathcal M)|
q^{(\chi(\mathcal M)-f(\mathcal M))n}
\bigl(1+O(q^{-n})\bigr).
\end{align}
\end{cor}

\begin{proof}
By \eqref{derivation phi},
\[
\Phi_n(\mathcal N_1,\mathcal N_2)
=
\sum_{\mathcal N_1\leq\mathcal Q\leq\mathcal N_2}
R_n(\mathcal N_1,\mathcal Q).
\]
By Proposition~\ref{rightnonalgebraic}, the only nonzero terms are the
diagonal term $\mathcal Q=\mathcal N_1$ and those for which
$\mathcal Q/\mathcal N_1$ is algebraic. Corollary~\ref{oright} gives
\[
R_n(\mathcal N_1,\mathcal Q)
=
q^{-\chi(\mathcal Q)n}
\bigl(1+O(q^{-n})\bigr)
\]
for every such $\mathcal Q$. Since
\[
\chi(\mathcal Q)
=
\chi(\mathcal N_1)
+
\chi(\mathcal Q/\mathcal N_1),
\]
we obtain \eqref{phi-asymptotic}. The reformulation in terms of
$\operatorname{Crit}(\mathcal L)$ follows by collecting the terms with
minimal Euler characteristic.

Finally,
\[
\Lambda_{\mathcal M}(n)
=
q^{\chi(\mathcal M)n}
\Phi_n(\{0\},\mathcal M),
\]
and therefore \eqref{lambda-partial-asymptotic} follows from
\eqref{phi-asymptotic}.
\end{proof}

We conclude the section with the proof of
Theorem~\ref{teo:average}.

\begin{proof}[Proof of Theorem~\ref{teo:average}]
By Proposition~\ref{presentationpartial}, there exists a
finite-dimensional partial $R$-module $\mathcal M$ such that
\[
L\cong U(\mathcal M).
\]
By the universal property of $U(\mathcal M)$,
\[
\Lambda_L(n)
=
\Lambda_{U(\mathcal M)}(n)
=
\Lambda_{\mathcal M}(n).
\]
Proposition~\ref{p: first formula for L}, together with
\eqref{derivation phi}, shows that $\Lambda_L(n)$ agrees with a
rational function of $q^n$ for all sufficiently large $n$.

By Proposition~\ref{freealgebraic}, every algebraic $R$-submodule
$N$ of $U(\mathcal M)$ is generated by $N\cap M$. Moreover,
Corollary~\ref{submod2} identifies $N$ with
\[
U(\mathcal M^{N\cap M}).
\]
Consequently, the assignment
\[
N\longmapsto\mathcal M^{N\cap M}
\]
gives a bijection between $A(U(\mathcal M))$ and $A(\mathcal M)$.
Under this correspondence,
\[
U(\mathcal M)/N
\cong
U\bigl(\mathcal M/\mathcal M^{N\cap M}\bigr).
\]
In particular, $A(L)$ is finite and, by
Corollary~\ref{uM},
\begin{align*}
\Lambda_L(n)
&=
\Lambda_{U(\mathcal M)}(n)
=
\Lambda_{\mathcal M}(n)\\
&=
q^{\chi(\mathcal M)n}
+
\sum_{\mathcal N\in A(\mathcal M)}
q^{\chi(\mathcal M/\mathcal N)n}
\bigl(1+O(q^{-n})\bigr)\\
&=
q^{\chi(U(\mathcal M))n}
+
\sum_{N\in A(U(\mathcal M))}
q^{\chi(U(\mathcal M)/N)n}
\bigl(1+O(q^{-n})\bigr)\\
&=
q^{\chi(L)n}
+
\sum_{N\in A(L)}
q^{\chi(L/N)n}
\bigl(1+O(q^{-n})\bigr).
\end{align*}
\end{proof}

  \section{The calculation of $\mathbb E_w[{\rm fix}]$}
   In this section, we show how the calculations from the previous section recover formula~\eqref{crit2} for \(\mathbb{E}_w[\operatorname{fix}]\). Without loss of generality, we may assume that \(w\) is cyclically reduced.
   Let \(\L\) be the canonical partial module whose envelope is \[ \F_q[F]/(w-1)\F_q[F], \] and whose underlying 
   \(\F_q\)-vector space has a basis consisting of the prefixes of \(w\). 
   First, recall that \[ \mathbb{E}_w[\operatorname{fix}] = \Lambda_{\F_q[F]/(w-1)\F_q[F]}(n) = \Lambda_{\L}(n) = q^n\Phi_n(\{0\},\L). \] 
   We then obtain \begin{align} \Phi_n(\{0\},\L) &\myeq{\eqref{derivation phi} and Proposition~\ref{rightnonalgebraic}} R_n(\{0\},\{0\}) + \sum_{\substack{\Q\in A(\L)\\ \Q\leq \L}} R_n(\{0\},\Q) \\ &\myeq{Corollary~\ref{oright}} q^{-n} + R_n(\{0\},\L) + \lvert\CR_q^2(w)\rvert q^{-2n} + O(q^{-3n}). \end{align} 
   By~\eqref{derivation phi}, we have \[ R_n(\{0\},\L) = \sum_{\mathcal{N}\leq \L} C_n(\mathcal{N},\L). \] 
   Using Proposition~\ref{formula for C}, we obtain: 
   \begin{enumerate} 
   \item[(a)] \( C_n(\L,\L)=q^{-n}. \)
   \item[(b)] 
   \( C_n(\mathcal{N},\L)=O(q^{-3n}) \qquad \text{if } \{0\}<\mathcal{N}<\L. \)
   \item[(c)] \(C_n(\{0\},\L)=O(q^{-3n}). \) \end{enumerate} 
   To prove the last estimate, we consider the coefficient of \(q^{-n}\) in the sum 
   \[C_n(0,\L)=q^{-n} \sum_{\substack{ T\in\mathbb{T}_{\L}^{\mathcal F}\\ \langle T\rangle=\L }} \epsilon(T)q^{-\lvert T\rvert n}. \]
   This coefficient corresponds to generating sets of \(\L\) consisting of a single element. 
   By~\cite{EPS21}, every generator of \(\L\) is represented by a group element. 
   Observe that each such group element contributes once with sign \(+1\) and once with sign \(-1\). 
   Therefore, the coefficient of \(q^{-n}\) vanishes, and hence 
   \[ \sum_{\substack{ T\in\mathbb{T}_{\L}^{\mathcal F}\\ \langle T\rangle=\L }} \epsilon(T)q^{-\lvert T\rvert n} = O(q^{-2n}). \] 
   Consequently, \[ R_n(\{0\},\L) = q^{-n}+O(q^{-3n}). \] It follows that \[ \mathbb{E}_w[\operatorname{fix}] = q^n\Phi_n(\{0\},\L) = 2+\lvert\CR_q^2(w)\rvert q^{-n} +O(q^{-2n}), \] which recovers formula~\eqref{crit2}.

\end{document}